\documentclass[11pt]{amsart}
\usepackage{amsthm}
\usepackage{graphicx}
\title[A note on Pythagorean Triples]{A note on Pythagorean Triples}
\author[R. Amato]{Roberto Amato}
\address[R. Amato]{Department of Engineering, University of Messina, 98166 Messina - (Italy)}
\email{ramato@unime.it}
\numberwithin{equation}{section}



\newtheorem{theorem}{Theorem}[section]

\newtheorem{example}{Example}[section]

\begin{document}

\begin{abstract}
Some relations among Pythagorean triples are established. The main tool is a fundamental characterization of the Pythagorean triples throught a chatetus which allows to determine relationships with Pythagorean triples having the same chatetus raised to an integer power.
\end{abstract}

\maketitle
\section{Introduction}\label{Introduction}
Let $x$, $y$ and $z$ be positive integers satisfying
$$x^2+y^2=z^2.$$
Such a triple $(x,y,z)$ is called Pythagorean triple and if, in addition, $x$, $y$ and $z$ are co-prime, it is called primitive Pythagorean triple. Preliminarly, let us recall a recent novel formula that allows to obtain all Pythagorean triples as follows.

\begin{theorem}
	$(x, y, z)$ is a Pythagorean triple if and only if there exists $d \in C(x)$ such that
		\begin{equation}\label{1.1}
	x=x, \quad y= \dfrac{x^2}{2d}-\dfrac{d}{2}, \quad z=\dfrac{x^2}{2d}+\dfrac{d}{2}\,,
	\end{equation}
with $x$ positive integer, $x \geq 1$, and where
\begin{equation*}
C(x)=\left\{
\begin{array}{ll}
D(x), & \hbox{if $x$ is odd,} \\
\\
D(x) \cap P(x), & \hbox{if $x$ is even,}
\end{array}
\right.
\end{equation*}
with
\begin{equation*}
D(x)=\left\{ d \in {\mathbb{N}} \quad \text{such that $d \leq x$ and $d$ divisor of $x^2$}\right\}\,,
\end{equation*}
and if $x$ is even with $x = 2^n k$, $n \in \mathbb{N}$ and $k \geq 1$ odd fixed, with
\begin{equation*}
P(x)=\left\{ d \in {\mathbb{N}} \quad \text{such that $d=2^s l$, with $l$ divisor of $x^2$ and $s \in \{1, 2, \ldots, n-1\}$}\right\}.
\end{equation*}
\end{theorem}

We want to find relations between the primitive Pythagorean triple $(x,y,z)$ generated by any predeterminated $x$ positive odd integer using \eqref{1.1} and the primitive Pythagorean triple generated by $x^m$ with $m\in \mathbb{N}$ and $m \geq 2$. In this paper we take care of relations only for the case in  which the primitive triple $(x,y,z)$ is generated whith $d \in C(x)$ only with $d=1$ and the primitive triple $(x^m,y',z')$ is generated with $d_m \in C(x^m)$ only with $d_m=1$ obtaining formulas that give us $y'$ and $z'$ directly from $x$, $y$, $z$. This is the first step to investigate on other relations between Pythagorean triples.

\section{Results}\label{Results}
We have that the following theorem holds.

\begin{theorem}
	Let $(x, y, z)$ be the primitive Pythagorean triple generated by any predeterminated positive odd integer $x \geq 1$ using \eqref{1.1} with $z-y=d=1$ and let $(x^m,y',z')$ be the primitive  Pythagorean triple generated by $x^m$, $m\in \mathbb{N}$, $m \geq 2$, using \eqref{1.1} with $z'-y'=d_m=1$, we have the following formulas
	\begin{eqnarray}\label{2.1}
	y'&=&y\left[ 1+ \sum_{p=1}^{m-1}x^{2p}\right]\,,\nonumber\\
	&&\\
	 z'&=&y\left[ 1+ \sum_{p=1}^{m-1}x^{2p}\right]+1\,,\nonumber
	\end{eqnarray}
	for every $m \in \mathbb{N}$ and $m \geq 2$.
	
	\noindent Moreover we have also
	\begin{equation} \label{2.2}
z\left[ (-1)^{m-1}+ \sum_{p=1}^{m-1} (-1)^{m-1-p}x^{2p}\right]=	\left\{
	\begin{array}{ll}
	y'& \hbox{if $m$ is even,}\\
	z'& \hbox{if $m$ is odd,}\\
	\end{array}
	\right.
	\end{equation}
and 
	\begin{equation} \label{2.3}
z\left[ (-1)^{m-1}+ \sum_{p=1}^{m-1} (-1)^{m-1-p}x^{2p}\right] + (-1)^{m-2}= 	\left\{
\begin{array}{ll}
z'& \hbox{if $m$ is even,}\\
y'& \hbox{if $m$ is odd.}\\
\end{array}
\right.
\end{equation}
\end{theorem}

\begin{proof}
Let $x$ be a positive odd integer that we consider as $x=2n+1$, $n \in \mathbb{N}$, so that using \eqref{1.1} with $d=z-y=1$ it results the primitive Pythagorean triple
\begin{equation}\label{2.4}
  x=2n+1, \quad y=2n^2+2n, \quad z= 2n^2+2n+1,
\end{equation}
while considering $x^m$, $m \in \mathbb{N}$, $m \geq 2$, using \eqref{1.1} with $d_m=z'-y'=1$ it results the primitive Pythagorean triple
\begin{equation}\label{2.5}
x^m, \quad y'= \dfrac{x^{2m}-1}{2}, \quad z'=\dfrac{x^{2m}+1}{2}\,.
\end{equation}
From the comparison between \eqref{2.4} and \eqref{2.5} we obtain
\begin{eqnarray*}
y'&=&\dfrac{(2n+1)^{2m}-1}{2}=\dfrac{\left[ (2n+1)^{2}-1\right] }{2}\left[(2n+1)^{2(m-1)}+ (2n+1)^{2(m-2)}+ \ldots +1\right] \\
&=&\dfrac{\left( 4n^2+4n\right) }{2}\left[1+ \sum_{p=1}^{m-1} (2n+1)^{2p}\right]=(2n^2+2n)\left[1+ \sum_{p=1}^{m-1} (2n+1)^{2p}\right]=y\left[1+ \sum_{p=1}^{m-1}x^{2p}\right]\,,
\end{eqnarray*}
that is the first of \eqref{2.1}, and because $d_m=z'-y'=1$ we obtain also
\begin{equation*}
z'=y\left[1+ \sum_{p=1}^{m-1}x^{2p}\right]+1\,,
\end{equation*}
that is the second of \eqref{2.1}.

\noindent Moreover, if $m$ is odd, using \eqref{2.4} and \eqref{2.5} we obtain
\begin{eqnarray*}
	z'&=&\dfrac{(2n+1)^{2m}+1}{2}=\dfrac{\left[ (2n+1)^{2}+1\right] }{2}\left[(2n+1)^{2(m-1)}- (2n+1)^{2(m-2)}+ \ldots -(2n+1)^2+1\right] \\
	&=&(2n^2+2n+1)\left[1+ \sum_{p=1}^{m-1} (-1)^{m-1-p}(2n+1)^{2p}\right]=z\left[1+ \sum_{p=1}^{m-1}(-1)^{m-1-p}x^{2p}\right]\,,
\end{eqnarray*}
that is the second case of \eqref{2.2}, and because $d_m=z'-y'=1$ we obtain also
\begin{equation*}
y'=z\left[1+ \sum_{p=1}^{m-1}(-1)^{m-1-p}x^{2p}\right]-1\,,
\end{equation*}
that is the second case of \eqref{2.3}.

\noindent At last, if $m$ is even, we proof that we obtain
\begin{equation}\label{2.6}
y'=z\left[-1+ \sum_{p=1}^{m-1}(-1)^{m-1-p}x^{2p}\right]\,,
\end{equation}
that is the first case of \eqref{2.2} and because $d_m=z'-y'=1$ we obtain also
\begin{equation*}
z'=z\left[-1+ \sum_{p=1}^{m-1-p}(-1)^{m-1-p}x^{2p}\right]+1\,,
\end{equation*}
that is the first case of \eqref{2.3}.

\noindent In order to prove that \eqref{2.6} holds we can write it using \eqref{2.4} and \eqref{2.5} in the following way
\begin{equation}\label{2.7}
	\dfrac{(2n+1)^{2m}-1}{2}=(2n^2+2n+1)\left[-1+ \sum_{p=1}^{m-1} (-1)^{m-1-p}(2n+1)^{2p}\right]\,,
\end{equation}
and we prove that \eqref{2.7} is an identity. In fact
\begin{equation*}
(2n+1)^{2m}-1=(4n^2+4n+2)\left[-1+ \sum_{p=1}^{m-1} (-1)^{m-1-p}(2n+1)^{2p}\right]\,,
\end{equation*}
\begin{equation*}
(2n+1)^{2m}-1=\left[ (2n+1)^2+1)\right] \left[-1+ \sum_{p=1}^{m-1} (-1)^{m-1-p}(2n+1)^{2p}\right]\,,
\end{equation*}
\begin{equation*}
(2n+1)^{2m}=-(2n+1)^2+\sum_{p=1}^{m-1} (-1)^{m-1-p}(2n+1)^{2(p+1)}+\sum_{p=1}^{m-1} (-1)^{m-1-p}(2n+1)^{2p}\,,
\end{equation*}
\begin{eqnarray}\label{2.8}
(2n+1)^{2m}&=&-(2n+1)^2+\Big[(-1)^{m-2}(2n+1)^{4} +(-1)^{m-3}(2n+1)^{6}\nonumber \\
&+&(-1)^{m-4}(2n+1)^{8}+ \ldots-(2n+1)^{2(m-1)}+ (2n+1)^{2m}\Big]\nonumber \\
&+& \Big[(-1)^{m-2}(2n+1)^{2} +(-1)^{m-3}(2n+1)^{4}+ (-1)^{m-4}(2n+1)^{6}\nonumber \\
&+& (-1)^{m-5}(2n+1)^{8}+\ldots-(2n+1)^{2(m-2)}+ (2n+1)^{2(m-1)}\Big]\,,
\end{eqnarray}
and because $m$ is even, after simplifying \eqref{2.8} we get
$$(2n-1)^{2m}=(2n-1)^{2m}\,,$$
so we prove that \eqref{2.7} is an identity and, therefore, \eqref{2.6} holds. 
Consequently, formulas \eqref{2.1}, \eqref{2.2} and \eqref{2.3} have thus been proved.

\noindent Obviously, because $z-y=d=1$, then we can obtain also other relations between $(x,y,z)$ and $(x^m,y',z')$, for example \eqref{2.1} are equivalent to
\begin{equation*}
y'=z+ y\sum_{p=1}^{m-1}x^{2p}-1\,,
\end{equation*}
\begin{equation*}
z'=z+ y\sum_{p=1}^{m-1}x^{2p}\,.
\end{equation*}
By similar way we can obtain other relations from \eqref{2.2} and \eqref{2.3}.
\end{proof}

To prove the accuracy of formulas \eqref{2.1}, \eqref{2.2} and \eqref{2.3} we consider the following example.

\begin{example}
	
	{\em We give the following table that can be extended for each primitive triples $x$, $y$, $z$, and $x^s$, $y'$, $z'$ with $x-y=1$ and $x'-y'=1$.
		
		\noindent Usung \eqref{2.1} we obtain
			\begin{center}
			\begin{tabular}{|p{1,1cm}|p{7cm}|p{2cm}|}
				\hline
				$x=3$ & $y=4$ &$z=5$ \\
				$x=3^2$ & $y'=4(1+3^2)=40$ &$z'=41$ \\
				$x=3^3$ & $y'=4(1+3^2+3^4)=364$ &$z'=365$ \\
				$x=3^4$ & $y'=4(1+3^2+3^4+3^6)=3280$ &$z'=3281$ \\
				$x=3^5$ & $y'=4(1+3^2+3^4+3^6+3^8)=29524$ &$z'=29525$ \\
				$x=3^6$ & $y'=4(1+3^2+3^4+3^6+3^8+3^{10})=265720$ &$z'=265721$ \\
				$\vdots$ & $\vdots$ & $\vdots$\\
					 		\hline
	 \end{tabular}
 \end{center}
While, using \eqref{2.2} and \eqref{2.3} we obtain
			\begin{center}
		\begin{tabular}{|p{1,1cm}|p{7cm}|p{2cm}|}
			\hline
			$x=3$ & $y=4$ &$z=5$ \\
			$x=3^2$ & $y'=5(-1+3^2)=40$ &$z'=41$ \\
			$x=3^3$ & $z'=5(1-3^2+3^4)=365$ &$y'=364$ \\
			$x=3^4$ & $y'=5(-1+3^2-3^4+3^6)=3280$ &$z'=3281$ \\
			$x=3^5$ & $z'=5(1-3^2+3^4-3^6+3^8)=29525$ &$y'=29524$ \\
			$x=3^6$ & $y'=5(-1+3^2-3^4+3^6-3^8+3^{10})=265720$ &$z'=265721$ \\
			$\vdots$ & $\vdots$ & $\vdots$\\
			\hline
		\end{tabular}
\end{center}}
\end{example}


\begin{thebibliography}{99}
	\bibitem{Amato} R. Amato, {\em A characterization of pythagorean triples}, JP Journal of Algebra, Number Theory and Applications {\bf 39} (2017), 221--230
\bibitem{Sierpinski} W. Sierpinski, {\em Elementary theory of numbers}, PWN-Polish Scientific Publishers, 1988.
\end{thebibliography}
\end{document}